\documentclass{amsart}
\usepackage{amsmath}
\usepackage{amssymb}
\usepackage{extarrows}
\usepackage{mathabx}
\usepackage[all]{xy}
\usepackage{amsbsy}
\usepackage{graphicx}
\usepackage{appendix}
\usepackage{float}
\usepackage{subfigure}
\usepackage[numbers,sort&compress]{natbib}
\usepackage{graphicx}
\usepackage{amsthm}
\usepackage{geometry}
\usepackage{fancyhdr}
\usepackage{color} 
\usepackage{overpic}
\usepackage{mathabx}
\usepackage{tikz-cd}
\usepackage{enumerate}
\usepackage{mathrsfs}
\usepackage{colonequals}
\usepackage{microtype}
\usepackage{hyperref}
\usepackage{diagbox}

\definecolor{lygreen}{HTML}{016646}

\newtheorem{thm}{Theorem}[section]
\newtheorem{cor}[thm]{Corollary}
\newtheorem{prop}[thm]{Proposition}
\newtheorem{lem}[thm]{Lemma}

\theoremstyle{definition}
\newtheorem{defn}[thm]{Definition}
\newtheorem{cons}[thm]{Construction}
\newtheorem{exmp}[thm]{Example}

\newtheorem*{fact}{Fact}

\newtheorem*{ack}{Acknowledgement}

\theoremstyle{remark}
\newtheorem{rem}[thm]{Remark}

\numberwithin{equation}{section}

\newcommand{\beq}{\begin{equation*}\begin{aligned}}
\newcommand{\eeq}{\end{aligned}\end{equation*}}
\newcommand{\bpf}{\begin{proof}}
\newcommand{\epf}{\end{proof}}
\newcommand{\bthm}{\begin{thm}}
\newcommand{\ethm}{\end{thm}}
\newcommand{\bprop}{\begin{prop}}
\newcommand{\eprop}{\end{prop}}
\newcommand{\bcor}{\begin{cor}}
\newcommand{\ecor}{\end{cor}}
\newcommand{\blem}{\begin{lem}}
\newcommand{\elem}{\end{lem}}
\newcommand{\bdefn}{\begin{defn}}
\newcommand{\edefn}{\end{defn}}
\newcommand{\bcons}{\begin{cons}}
\newcommand{\econs}{\end{cons}}
\newcommand{\bexmp}{\begin{exmp}}
\newcommand{\eexmp}{\end{exmp}}
\newcommand{\brem}{\begin{rem}}
\newcommand{\erem}{\end{rem}}
\newcommand{\bfa}{\begin{fact}}
\newcommand{\efa}{\end{fact}}
\newcommand{\benu}{\begin{enumerate}[(1)]}
\newcommand{\eenu}{\end{enumerate}}

\newcommand{\bdia}{\begin{displaymath}\xymatrix}
\newcommand{\edia}{\end{displaymath}}

\newcommand{\al}{\alpha}

\newcommand{\ga}{\gamma}

\newcommand{\intg}{\mathbb{Z}}

\newcommand{\tb}{\text{tb}}

\DeclareMathOperator{\rot}{rot}

\begin{document}

 \title{On Legendrian representatives of non-fibered knots}


\author{Zhenkun Li}
\address{Department of Mathematics and Statistics, University of South Florida}
\curraddr{}
\email{zhenkun@usf.edu}
\thanks{}

\author{Shunyu Wan}
\address{School of Mathematics, Georgia Institute of Technology}
\curraddr{}
\email{swan48@gatech.edu}
\thanks{}

\date{}
\dedicatory{}
\maketitle
\begin{abstract}
We show that in $(S^3,\xi_{std})$ if $K$ is a non-trivial knot that realizes the three-dimensional Thurston-Bennequin bound (i.e. $K$ has a Legendrian representative $\Lambda$ with $\tb(\Lambda)-\rot(\Lambda)=2g(K)-1$), then $K$ has a Legendrian representative $L$ with $\tb=0$. Moreover, this result can be easily generalized to contact manifolds that uniquely represent the associated contact invariants. This is the first result on Legendrian representatives of non-fibered knots in $3-$manifolds other than $S^3$. We also show that if $K$ is a nearly fibered knot in $S^3$ then $\tau(K)=g(K)$ implies that $K$ realizes the three-dimensional Thurston-Bennequin bound. 
\end{abstract}

\section{Introduction}
Fibered knots are closely related to the study of contact geometry in dimension $3$, since the fibration of the knot complement naturally gives rise to an open book for the contact structure. Hence, we know several general contact properties of fibered knots (for example, in \cite{EtnyreVanHMFiberedBennequin},\cite{hedden2010positivity}). In this paper, we will investigate contact properties of non-fibered knots, which were previously less understood, using techniques from the suture Floer homology and the convex surface decomposition. 

Let us first recall some standard terminologies. Let $K$ be a smooth null-homologous knot in a $3-$manifold $Y$; we denote $g(K)$ as the genus of $K$. Moreover, if $\xi$ is a contact structure on $Y$, we denote $\overline{\tb}_\xi(K)$ the maximal $\tb$ number of $K$ and $\overline{\text{sl}}_\xi(K)$ the maximal self linking number of $K$ in $(Y,\xi)$. We say $K$ in $Y$ realizes the three-dimensional Thurston-Bennequin bound in $(Y,\xi)$ if $K$ has a Legendrian representative $\Lambda$ in $(Y,\xi)$ satisfying $\tb(\Lambda)-\rot(\Lambda)=2g(K)-1$, or equivalently $\overline{\text{sl}}_\xi(K)=2g(K)-1$. If we talk about a knot $K$ in $S^3$, we only consider the unique tight contact structure in $S^3$, and the above terminologies will be abbreviated to $\overline{\tb}(K)$, $\overline{\text{sl}}(K)$ and $K$ realizes the three-dimensional Thurston-Bennequin bound. We will first look at results for knots in $S^3$.

\begin{thm} \label{thm: tb=0 for knot realizeing T-B bound}
    If $K$ is a non-trivial knot in $S^3$ that realizes the three-dimensional Thurston-Bennequin bound, then $\overline{\rm tb}(K) \geq 0$.
\end{thm}

\begin{rem}
    Theorem \ref{thm: tb=0 for knot realizeing T-B bound} is not true if instead we assume that $K$ realizes the slice-Bennequin bound, that is, $K$ admits a Legendrian representative $L$ such that $\tb(L)-{\rm rot}(L)=2g_4(K)-1$. For example the knot $8_{20}$ is quasi-positive, so it realizes the slice-Bennequin bound but its maximal $\tb$ number is $-2$.
\end{rem}

Note the fact that $K$ realizes the three-dimensional Thurston-Bennequin bound implies that $\tau(K)=g(K)$ as in \cite{OlgaBoundsfortheTB}, where $\tau$ is the Heegaard Floer $\tau$ invariant. Thus, when $K$ is fibered by \cite[Theorem 1.2]{hedden2010positivity}, we know $K$ is strongly quasi-positive; hence, by \cite[Proposition 2]{RudolphObstructiontoSliceness}, when $K$ is non-trivial, we have $\overline{\rm tb}(K) \geq 0$. In general, there are ways to find a $\tb=0$ representative for any fibered knots in the contact manifold supported by such a fibered knot by  manipulating the Heegaard surface from the open book. (See the forthcoming paper by Etnyre, Li, and Tosun\cite{ELTforthcoming}). However, the non-fibered knots are less understood as there is no natural open book associated with them.


Since strongly quasi-positive knots always realize the three-dimensional Thurston-Bennequin bound, we give an alternative proof of the following corollary, which was originally proved by Rudolph \cite{RudolphObstructiontoSliceness}. 

\begin{cor}
If $K$ is a non-trivial strongly quasi-positive knot, then $\overline{\rm tb}(K)\geq 0$. 
\end{cor}

Moreover, once we have $tb=0$ Legendrian representatives it's easy to obtain the following corollary. 

\begin{cor}
    If $K$ in $S^3$ realizes the three-dimensional Thurston-Bennequin bound, then $S^3_{-1}(K)$ admits a Stein fillable contact structure.
\end{cor}

\bpf
When $K=U$ is the trivial unknot, $S^3_{-1}(U)=S^3$, which is Stein fillable. When $K$ is non-trivial, we can perform the Legendrian surgery ($-1$ contact surgery) on the $tb=0$ Legendrian representative, which is smoothly $-1$ surgery as well. Since Legendrian surgery preserves Stein fillability, $S^3_{-1}(K)$ admits a Stein fillable contact structure.
\epf

The above corollary can be compared with Theorem 1.6 in \cite{ozsvath2005contact}, proved by Ozsv\'ath and Szab\'o, which says that if $K$ is fibered in $S^3$, then $S^3_{-1}(K)$ admits a tight contact structure with a non-vanishing contact invariant. 

For nearly fibered knots, since we know exactly what the topology of the knot complements looks like \cite{LY2023nearly}, we can say more about when the nearly fibered knots realize the three-dimensional Thurston-Bennequin bound. Thus, Theorem \ref{thm: tb=0 for knot realizeing T-B bound} is stronger for nearly fibered knots.  

\begin{thm} \label{thm: nearly fibered knot}
    Let $K$ in $S^3$ be a nearly fibered knot, {\it i.e.}
	\[
	\widehat{HFK}(S^3,K,g(K)) \cong \mathbb{F}^2.
	\]
	Suppose further that $\tau(K) = g(K)$, then $\overline{sl}(K)=2g(K)-1$ (i.e. K realizes the three-dimensional Thurston-Bennequin bound) and $\overline{tb}(K)\geq 0$, 
\end{thm}



It is probably also worth mentioning the following easy corollary as it is interesting to see how contact properties of knots can give restrictions on the geography of knot Floer. 

\begin{cor}
    If a knot $K$ in $S^3$ satisfies $\tau(K)=g(K)$ but does not realize the three-dimensional Thurston-Bennequin bound, then $\text{rank}(\widehat{HFK}(K,g(K)))\geq 3$.
\end{cor}
\begin{proof}
    \cite[Theorem 1.2]{hedden2010positivity} says $K$ is not fibered, and Theorem \ref{thm: nearly fibered knot} says $K$ is not nearly fibered so the rank of the top grading has to be at least $3$. 
\end{proof}

The proofs of Theorem \ref{thm: tb=0 for knot realizeing T-B bound} and Theorem \ref{thm: nearly fibered knot} rely on the fact that there is a unique tight contact structure on $S^3$ with non-vanishing contact invariant. Thus, whenever we obtain a tight contact structure on $S^3$ from the knot complement, it must be the standard one. This trick does not always apply to general contact $3$-manifolds. However, if the knot is a fibered knot in $Y$ and we consider the contact structure $\xi_K$ associated to $K$, it is not hard to show the contact structure we construct is actually the same as $\xi_K$ (up to adding Giroux torsion). Thus, we have the following corollary. 

\begin{cor} \label{thm: Tight contact strucutre from fibered tau}
If a fibered knot $K$ in $Y$ is non-trivial and the Heegaard Floer contact invariant $c(\xi_K)\neq 0$, then $\overline{\tb}_{\xi_K}(K) \geq 0$ in $(Y,\xi_K)$.
\end{cor}

\begin{rem}
    If a fibered knot $K$ is trivial, i.e. $K$ bounds a disk in $Y$, it is straightforward to check this implies $(Y,K)=(S^3,U)$, where $U$ is the unknot. 
 \end{rem}

After we see the proof of Theorem \ref{thm: tb=0 for knot realizeing T-B bound} in section \ref{subsec: Knots realizing Thurston-Bennequin bound}, it directly extends to the following special family of contact manifolds.

\bdefn Let $(Y,\xi)$ be a 
 contact $3-$manifold, and $c(\xi)$ be the Heegaard Floer contact invariant of $(Y,\xi)$. We say $(Y,\xi)$ is \textbf{uniquely represented} if $(Y,\xi)$ \textbf{uniquely represents} $c(\xi)$, that is,  there is no other contact structure $\xi'$ on $Y$ such that $c(\xi')=c(\xi)$. 
 
 Note that $(Y,\xi)$ being uniquely represented implies $\xi$ is tight and $c(\xi)\neq 0$ because all different overtwisted contact structures have trivial contact invariants.
\edefn

 \bthm \label{thm: tb=0 for 2g-1 in other 3-manifold}
Let $K$ be a null-homologous knot in a uniquely represented $(Y,\xi)$. If $K$ realizes the three-dimensional Thurston-Bennequin bound in $(Y,\xi)$ then $\overline{\tb}_\xi(K)\geq 0$. 
 \ethm
\begin{rem}
    We conjecture that $\overline{\tb}_\xi(K) \geq 1$ for $K$ non-trivial and realizes the three-dimensional Thurston-Bennequin bound in any $(Y,\xi)$, but that would require us to understand the non-top grading of a relevant sutured Floer homology, which is beyond the techniques used in the current paper.
\end{rem}
 \begin{rem}
    Theorem \ref{thm: nearly fibered knot} is more subtle to generalize as it requires studying the explicit topology of the complement of nearly fibered knots in other $3-$manifolds. 
\end{rem}

There are actually many uniquely represented contact manifolds because contact invariants may completely classify the tight contact structures on certain $3-$manifolds (for example, see \cite{TosunTightsmallSeifertfiberedmanifolds}, \cite{GhigginiLiscaStipsiczTightonSSF}, \cite{ConwayMinTightstrucutreonfigureeight}, \cite{GhigginiVanHornMorrisTightContactStructuresOnBrieskornSpheres}, \cite{ShunyuTightcontactstructure}, \cite{MatkovicTightonLspace}, \cite{GhigginiTightwithmaximaltwisting}). The most famous one other than $S^3$ might be $\Sigma(2,3,5)$, the Poincar\'e homology sphere that is an integer homology sphere and admits a unique tight contact structure just like $S^3$. Thus, when considering this unique tight contact structure, we obtain the following corollary.

\bcor
If $K$ is a non-trivial knot in $\Sigma(2,3,5)$ that realizes the three-dimensional Thurston-Bennequin bound, then $\overline{\tb}(K) \geq 0$.
\ecor

\vspace{0.2in}
\noindent
{\bf Strategy of the proofs}. Here we explain the ideas behind the construction of $\tb=0$ representatives in various cases in this paper. We start with the contact $3$-manifold $(Y,\xi)$ and a null-homologous knot $K\subset Y$.
On the knot complement, we take the suture $\Gamma_0$, which consists of two longitudes of the knot on $\partial Y\backslash N(K)$. This forms a balanced sutured manifold $(Y\backslash N(K),\Gamma_0)$. If we attach a contact $2$-handle along the meridian of the knot, the resulting sutured manifold is $Y(1)=(Y\backslash B^3,\delta)$ for a connected suture $\delta$. The contact $2$-handle attachment induces a map ({\it c.f.} \cite{juhasz1803contact})
\[
	\pi^0:SFH(-Y\backslash N(K),-\Gamma_0) \to SFH(-Y(1))=\widehat{HF}(-Y).
\]
When $(Y,\xi)$ is uniquely represented (For example $(S^3,\xi_{std})$), admitting a $\tb=0$ representative is equivalent to the fact that there exists a contact structure $\xi^0_K$ on $(Y\backslash N(K),\Gamma_0)$ such that
\begin{equation}\label{eq: intro, xi_K}
	\pi^0(c(\xi^0_K)) = c(\xi).
\end{equation}
We further make use of the following commutative diagram [{\it c.f.} Equation (\ref{eq: pi map})]
\[
\xymatrix{
SFH(-(Y\backslash N(K)),-\Gamma_0,g(K)-\frac{1}{2})\ar[rr]^{\psi_{-,1}^0}\ar[dr]^{\pi^0}&&SFH(-(Y\backslash N(K)),-\Gamma_{1},g(K))\ar[dl]^{\pi^1}\\
&SFH(-Y(1))=\widehat{HF}(-Y)&
}
\]
together with the fact that $\psi^0_{-,1}$ is surjective on the top grading  when $K$ is non-trivial ({\it c.f.} Lemma \ref{lem: psi^0_1 is surjective on top grading}). Note here the sutured Floer homologies of the sutured knot complements admit a grading induced by the Seifert surface of the knot, under which $g(K)-\frac{1}{2}$ and $g(K)$ are the top non-vanishing grading of the respective sutured Floer homology groups. 

If the knot $K\subset S^3$ satisfies the condition $\tau(K) = g(K)$, we find that the map $\pi^0$ is non-trivial ({\it c.f.} proof of Lemma \ref{lem: pi^0 is non-trivial on top grading}). When $K$ is nearly fibered or fibered, we know models for the Seifert surface complement of the knot, and hence we will show that $SFH(-(S^3\backslash N(K)),-\Gamma_0,g(K)-\frac{1}{2})$ is fully generated by contact elements. Thus, at least one contact element is non-trivial under the map $\pi^0=\pi^1\circ \psi^0_{-,1}$. Since contact handle maps automatically preserve contact elements and there is a unique tight contact structure on $S^3$, this completes the proof of Theorem \ref{thm: nearly fibered knot} and Corollary \ref{thm: Tight contact strucutre from fibered tau}.

For a knot $K$ realizing the three-dimensional Thurston-Bennequin bound, we apply similar ideas. Even though we still have $\tau(K)=g(K)$, it is not necessarily true that $SFH(-(S^3\backslash N(K)),-\Gamma_0,g(K)-\frac{1}{2})$ is fully generated by contact elements; hence, we need to construct the contact structure $\xi_K^0$ explicitly. To do this, we observe that $K$ at least has a Legendrian representative with $\tb=-n$ for some large enough positive integer $n$ that also realizes the three-dimensional Thurston-Bennequin bound. Now we construct the desired contact structure $\xi^0_K$ as in Equation \ref{eq: intro, xi_K} as follows: Note there are sutured manifold decompositions along a minimal-genus Seifert surface $S$ of $K$:
\[
	(S^3\backslash N(K),\Gamma_0)\stackrel{S}{\leadsto} (S^3\backslash N(S),\gamma^3) \stackrel{A}{\leadsto} (S^3\backslash N(S),\gamma^1)\cup (V,\gamma^4)
\]
and
\[
	(S^3\backslash N(K),\Gamma_n)\stackrel{S}{\leadsto} (S^3\backslash N(S),\gamma^1),
\]
where $(S^3\backslash N(S),\gamma^3)$ and $(S^3\backslash N(S),\gamma^1)$ are $S^3\backslash N(S)$ with $3$ and $1$ boundary-parallel sutures respectively, $A$ is an annulus as a push off of neighborhood of $\partial S$, and $(V,\gamma^4)$ is the solid torus with $4$ longitudinal sutures.

The $\tb=-n$ representative of $K$ gives rise to a contact structure $\xi_{K}^n$ on $(S^3\backslash N(K),\Gamma_n)$. We then use the convex decomposition introduced by Honda \cite{honda2000classification} to decompose $\xi^n_K$ into a contact structure $\xi_S$ on $(S^3\backslash N(S),\gamma^1)$. The fact that $K$ fulfills the three-dimensional Thurston-Bennequin equality ensures that the sutured manifold decomposition and convex surface decomposition are compatible with each other, meaning the grading on the contact invariants are matched up. Then we re-glue $\xi_S$ [together with a suitably chosen contact structure on $(V,\gamma^4)$] along the two-step decomposition to obtain the desired tight contact structure $\xi_K^0$ on $(S^3\backslash N(K),\Gamma_0)$.

\vspace{0.2in}
\noindent
{\bf Organization}. The paper is organized as follows. In Section \ref{sec: Preliminary}, we first recall some background on sutured manifolds, Floer homology, and contact invariants. Then we prove important lemmas and propositions (\ref{prop: product disk decomposition and HKM's gluing}, \ref{lem: solid torus with four longitudes }, \ref{lem: psi^0_1 is surjective on top grading}, \ref{lem: pi^0 is non-trivial on top grading}) 
 that we need to use later. In Section \ref{sec: constructing tb}, we will first prove Theorem \ref{thm: nearly fibered knot}, and then, after we establish the proof for the nearly fibered case on how to construct the $\tb=0$ representatives we will extend the strategy to prove Theorem \ref{thm: tb=0 for knot realizeing T-B bound}. Lastly, in Section \ref{subsection: Fibered knots}, we use a similar idea to prove Corollary \ref{thm: Tight contact strucutre from fibered tau} about fibered knots in other $3-$manifolds and Theorem $\ref{thm: tb=0 for 2g-1 in other 3-manifold}$.

\begin{ack}
The authors extend their gratitude to John Etnyre, Tom Mark, Gordana Matic, Steven Sivek and Bülent Tosun for their valuable suggestions. Additionally, we thank the organizers of the conference “Recent Developments in 3- and 4-Manifold Topology,” held in Princeton, for facilitating this collaboration. Shunyu Wan’s work was supported in part by grants from the NSF (RTG grant DMS-1839968) and the Simons Foundation (grants 523795 and 961391 to Thomas Mark).
\end{ack}

\section{Preliminary} \label{sec: Preliminary}
In this section, we review some background on sutured manifolds, Floer homology, and their relation to contact geometry.

\subsection{Balanced sutured manifolds}
Suppose $M$ is a compact oriented $3$-manifold with non-empty boundary. Let $\gamma=(A(\gamma),s(\gamma))$ be a pair where $A(\gamma)\subset\partial M$ is a disjoint union of embedded annuli on $\partial M$ and $s(\gamma)$ is the core of $A(\gamma)$ and is oriented. Since $\partial A(\gamma)$ is parallel to $s(\gamma)$, we orient $\partial A(\gamma)$ according to the given orientation of $s(\gamma)$. We can decompose $R(\gamma) = \overline{\partial M - A(\gamma)}$ into two parts:
\[
R(\gamma) = R_+(\gamma) \cup R_-(\gamma),
\]
where $R_+(\gamma)$ is the part of $R(\gamma)$ such that the orientation induced by $\partial M$ coincides with the orientation induced by $\partial R(\gamma) = \partial A(\gamma)$, and $R_-(\gamma) = R(\gamma) - R_+(\gamma)$.

\begin{rem}
    Our definition of the suture $\gamma$ is more restrictive than the original definition in \cite{gabai1983foliations}, in the sense that we exclude the possibility that $\gamma$ can contain some toroidal components, as such components cannot exist for balanced sutured manifolds. Since $A(\gamma)$ is simply a tubular neighborhood of $s(\gamma)$, in some occasions we will just use $s(\gamma)$ in place of $\gamma$.
\end{rem}

\begin{defn}[{\cite[Definition 2.2]{juhasz2006holomorphic}}]
Suppose $(M,\gamma)$ is a pair as described above. If $M$ is connected, then we call $(M,\gamma)$ a {\bf balanced sutured manifold} if the following hold.
\begin{itemize}
	\item For every component $V\subset\partial M$, $V\cap A(\gamma)\neq \emptyset$.
	\item We have $\chi(R_+(\gamma)) = \chi(R_-(\gamma)).$
\end{itemize}
If $M$ is disconnected, then we call $(M,\gamma)$ a balanced sutured manifold if for any component $C\subset M$, $(C,C\cap\gamma)$ is a balanced sutured manifold in the above sense.
\end{defn}

\begin{defn}
	A balanced sutured manifold is called {\bf taut} if the following hold.
	\begin{itemize}
		\item $M$ is irreducible.
		\item $R(\gamma)$ is incompressible.
		\item $R_{+}(\gamma)$ and $R_{-}(\gamma)$ are both Thurston-norm minimizing in $H_2(M,s(\gamma);\intg)$.
	\end{itemize}
\end{defn}

\begin{defn}[{\cite[Definition 2.6]{juhasz2010polytope}}]\label{defn: well-groomed}
	Suppose $(M,\gamma)$ is a balanced sutured manifold and $S\subset M$ is a properly embedded surface. We say $S$ is {\bf well-groomed} if the following hold.
	\begin{itemize}
		\item Any component of $\partial S$ does not bound a disk inside $R(\gamma)$
		\item No component of $S$ is a disk with $\partial S\subset R(\gamma)$.
		\item For any component $\alpha$ of $\partial S\cap A(\gamma)$, one of the following two conditions holds:
		\begin{itemize}
			\item If $\alpha$ is an arc, then $\alpha$ is properly embedded and non-separating, with $|\alpha\cap s(\gamma)| = 1$.
			\item If $\alpha$ is a closed curve, and $A$ is the component of $A(\gamma)$ that contains $\alpha$, then $\alpha$ is a simple closed curve so that $[\alpha] = [A\cap s(\gamma)]\in H_1(A;\intg).$
		\end{itemize}
		\item For any component $V$ of $R(\gamma)$, one of the following two conditions holds:
		\begin{itemize}
			\item If $V$ is planar, then $\partial S\cap V$ is a union of parallel, coherently oriented, non-separating simple closed arcs
			\item If $V$ is non-planar, then $\partial S\cap V$ is a union of parallel, coherently oriented, non-separating simple closed curves.
		\end{itemize}
	\end{itemize}
\end{defn}

In \cite{juhasz2006holomorphic}, Juh\'asz constructed sutured Floer homology for any balanced sutured manifold $(M,\gamma)$, which we denote by $SFH(M,\gamma)$. In this paper, we use $\mathbb{F} = \intg\slash(2\intg)$ coefficient.

\bprop[{\cite[Theorem 1.3]{juhasz2008floer}}] \label{prop: Juh well-groomed sfh map}
	Suppose $(M,\gamma)$ is a balanced sutured manifold and $S\subset (M,\gamma)$ is a well-groomed surface. Let \[
		(M,\ga)\stackrel{S}{\leadsto}(M',\ga')
	\]
	be the sutured manifold decomposition. Then there is an injective map
	\[
	\iota_S: SFH(M',\gamma')\to SFH(M,\gamma).
	\]
\eprop

\bdefn
Suppose $(M,\gamma)$ is a balanced sutured manifold. A properly embedded disk $D\subset M$ is called a {\bf product disk} if $|D\cap s(\gamma)| = 2$. A properly embedded annulus $A\subset M$ is called a {\bf product annulus} if $\partial A\subset R(\gamma)$, $\partial A\cap R_+(\gamma)\neq\emptyset$ and $\partial A\cap R_-(\gamma)\neq\emptyset$.
\edefn

\bprop[\cite{juhasz2008floer}]\label{prop: product decomposition}
	Suppose $(M,\gamma)$ is a balanced sutured manifold and $S\subset (M,\gamma)$ is either a product disk or a well-groomed product annulus. Let
	\[
		(M,\ga)\stackrel{S}{\leadsto}(M',\ga')
	\]
	be the sutured manifold decomposition. Then
	\[
	SFH(M',\gamma')\cong SFH(M,\gamma).
	\]
\eprop

\subsection{Contact structures and contact elements}
Contact structures can be naturally adapted to the setup of a balanced sutured manifold.
\bdefn
Suppose $(M,\gamma)$ is a balanced sutured manifold. We say a co-oriented contact structure $\xi$ is compatible with $(M,\gamma)$ if $\partial M$ is convex and $\gamma$ is (isotopic to) the dividing curve on $\partial M$. 
\edefn

\bthm[{\cite[Theorem 0.1 and Corollary 4.3]{honda2009contact}}]
Suppose $\xi$ is a contact structure compatible with $(M,\gamma)$. Then there is a well-defined element
\[
c(\xi) \in SFH(-M,-\gamma).
\]
Furthermore, if $\xi$ is overtwisted, then $c(\xi) = 0$.
\ethm

Next, we state some results regarding how contact elements change through sutured manifold decompositions.

\bprop[{\cite[Theorem 6.2]{honda2009contact}}]\label{prop: surface decomp and HKM's gluing}
Suppose $(M,\gamma)$ is a balanced sutured manifold and $S\subset (M,\gamma)$ is a well-groomed surface. Let 
\[
	(-M,-\ga)\stackrel{S}{\leadsto}(-M',-\ga')
\]
be the sutured manifold decomposition. Let $\xi$ be a contact structure on $(M,\gamma)$ that is obtained from a contact element $\xi'$ on $(M',\gamma')$ by gluing along the convex surface $S$ with a boundary-parallel dividing set $\Gamma_S$. Let $c(\xi)\in SFH(-M,-\gamma)$ and $c(\xi')\in SFH(-M',-\gamma')$ be the contact elements, respectively. Then we have
\[\iota_S(c(\xi')) = c(\xi).\]
\eprop

\brem
The dividing curve $\Gamma_S$ is described explicitly in \cite[Page 3]{honda2002convex}. The subtlety is that, when we think about the convex decomposition of $(M,\gamma)$ along $S$ instead of the contact gluing of $(M',\gamma')$ along $S$, we may not be able to directly work with the well-groomed surface $S$ but in fact need to work with an isotopy $S'$ of $S$ such that each boundary component of $S'$ intersects $s(\gamma)$. \cite{honda2002convex} explained the construction of such an isotopy and the choice of dividing set on $S'$ in detail. The convex surface $S$ in the statement of above theorem should in fact, should be $S'$.
\erem

In addition to well-groomed surfaces, we also need to study the decomposition along product disks and annuli. If $S\subset (M,\gamma)$ is a product disk, then decomposing along $S$ can be viewed as the inverse operation of attaching a contact $1$-handle. The effect of attaching a contact $1$-handle has been studied in \cite[Section 3.3.2]{juhasz1803contact}. 
\bprop\label{prop: product disk decomposition and HKM's gluing}
Suppose $(M,\gamma)$ is a balanced sutured manifold and $D\subset (M,\gamma)$ is a product disk. Let 
\[
	(-M,-\ga)\stackrel{D}{\leadsto}(-M',-\ga')
\]
be the sutured manifold decomposition. Then we have the following.
\begin{itemize}
	\item The is an isomorphism
	\[
	\iota_{D} : SFH(-M',-\gamma') \xrightarrow{\cong} SFH(-M,-\gamma)
	\]
	\item If $\xi$ is a contact structure on $(M,\gamma)$, then we can perturb $D$ to be a convex surface, and there is a canonical dividing set $\Gamma_D$ on (the convex surface) $D$. Let $\xi'$ be the result of a convex decomposition of $\xi$ along $(D,\Gamma_D)$. Then we have
	\[
		\iota_D(c(\xi')) = c(\xi)
	\]
	\item If $\xi'$ is a contact structure on $(M',\gamma')$, then we can obtain a contact structure $\xi$ on $(M,\gamma)$ by gluing $\xi'$ along $(D,\Gamma_D)$, or equivalently, attaching a contact $1$-handle. Furthermore, we have
	 \[
		\iota_D(c(\xi')) = c(\xi)
	 \]
\end{itemize}
\eprop

\bcor\label{cor: unique contact structure on product sutured manifold}
Suppose $F$ is a compact oriented surface with $\partial F\neq \emptyset$. Let 
\[
(M,\gamma) = ([-1,1]\times F,\{1\}\times \partial F)
\]
be a product sutured manifold. Then there is a unique (up to isotopy) tight contact structure $\xi_{F}$ on $(M,\gamma)$, and 
\[
c(\xi_{F})\neq 0 \in SFH(-M,-\gamma) \cong \mathbb{F}.
\]
\ecor

\bpf
Without loss of generality we can assume that $F$ is connected. The uniqueness follows from the algorithm to classify tight contact structures on handlebodies developed by Honda in \cite{honda2002gluing}. When $\partial F\neq \emptyset$, we can find a set of pair-wise disjoint proper simple arcs that cut $F$ into a disk. Inside the product sutured manifold each such arc becomes a product disk (which intersects the suture $s(\gamma)$ twice). There is a unique choice of allowable dividing curves on each product disks. As a result of the classification algorithm, there is a unique tight contact structure $\xi_{F}$ on $(M,\gamma)$. The fact that $c(\xi_{F})\neq 0$ follows from Proposition \ref{prop: product disk decomposition and HKM's gluing}.  
\epf

We say a product annulus $A\subset (M,\gamma)$ is trivial if $\partial A$ bounds two disks $D_{\pm}\subset R_{\pm}(\gamma)$ and $D_{-}\cup A\cup D_+$ bounds a $3$-ball inside $M$. We believe the following proposition is known to experts, but since we cannot find a reference, we also provide a proof below.

\bprop\label{prop: product annulus decomposition and HKM's gluing}
Suppose $(M,\gamma)$ is a taut balanced sutured manifold and $S\subset (M,\gamma)$ is a non-trivial product annulus. Let 
\[
	(-M,-\ga)\stackrel{A}{\leadsto}(-M',-\ga')
\]
be the sutured manifold decomposition. Then there exists an injective map
\[
\iota_{A}: SFH(-M',-\gamma')\to SFH(-M,-\gamma).
\]
Furthermore, for any tight contact structure $\xi'$ on $(M',\gamma')$, there exists a contact structure $\xi$ on $(M,\gamma)$ such that
	\[
		\iota_A(c(\xi')) = c(\xi).
	\]
\eprop

\bpf
First we claim that no component of $\partial A$ bounds a disk on $R(\gamma)$. If both components of $\partial A$ bound disks $D_{\pm}$, then we have a $2$-sphere $D_-\cup A\cup D_+$. Since $(M,\gamma)$ is taut, $M$ is irreducible and so $D_-\cup A\cup D_+$ bounds a $3$-ball which implies that $A$ is trivial. Now assume, without loss of generality, $\partial A$ bounds a disk $D_-$ on $R_-(\gamma)$ while it does not bound a disk on $R_{+}(\gamma)$; then $D_-\cup A$ is a compressing disk of $R_+(\gamma)$, which contradicts to the fact that $(M,\gamma)$ is taut.

From now on we assume neither component of $\partial A$ bounds a disk on $R(\gamma)$. We study a few cases. 

{\bf Case 1}. Both components of $\partial A$ are non-separating on $R(\gamma)$. In this case $A\subset (M,\gamma)$ is a well-groomed surface so Proposition \ref{prop: surface decomp and HKM's gluing} applies. Note by \cite[Lemma 8.9]{juhasz2008floer} that in this case $\iota_A$ is an isomorphism.

{\bf Case 2}. $(M',\gamma')$ is the disjoint union of two pieces
\[
(M',\gamma') = (M_1,\gamma_1) \sqcup ([-1,1]\times F, \{0\}\times F),
\]
where $F$ is a connected compact oriented surface and $\partial F$ is disconnected. In this case, we can instead build $(M,\gamma)$ from $(M_1,\gamma_1)$ by attaching a sequence of contact $1$-handles. Indeed, in this case, one component $\alpha$ of $\partial F$ is glued to $\gamma_1$. But by assumption $\partial F - \alpha \neq \emptyset$, we can then choose a set of pair-wise disjoint properly embedded simple arcs that are all disjoint from $\alpha$ and which cut $F$ into an annulus. Then in the product sutured manifold $[-1,1]\times F$, each such arc corresponds to a product disk. Hence Proposition \ref{prop: product disk decomposition and HKM's gluing} applies. Note that in this case, $\iota_A$ is an isomorphism.

{\bf Case 3}. $(M',\gamma')$ is the disjoint union of two pieces
\[
(M',\gamma') = (M_1,\gamma_1) \sqcup ([-1,1]\times F, \{0\}\times F),
\]
where $F$ is a connected compact oriented surface and $\partial F$ is connected. Our assumption on $A$ makes sure that $F$ is not a disk, and hence it must have a positive genus. Thus we can pick a non-separating simple closed curve $\alpha$ on $F$ and let $F' = F-\alpha$. Suppose $\xi'$ is a tight contact structure on $(M',\gamma')$. It must be the disjoint union of a tight contact structure $\xi_1$ on $(M_1,\gamma_1)$ and $\xi_F$ on $([-1,1]\times F, \{0\}\times F)$, where $\xi_{F}$ is the unique tight contact structure on the product sutured manifold as in Corollary \ref{cor: unique contact structure on product sutured manifold}. Let $A'=[-1,1]\times \alpha$. We have a two steps decomposition:
\[
(M,\gamma)\stackrel{A'}{\leadsto}(M'',\gamma'')\stackrel{A}{\leadsto} (M_1,\gamma_1)\sqcup ([-1,1]\times F', \{0\}\times F')
\]
There is again a unique tight contact structure $\xi_{F'}$ on $([-1,1]\times F', \{0\}\times F')$. Case 2 applies to $\xi_1\sqcup \xi_{F'}$ to obtain a tight contact structure $\xi''$ on $(M'',\gamma'')$. Then the first decomposition along $A'$ above is well-groomed, so Case 1 applies. As a result, we obtain a tight contact structure $\xi$ on $(M,\gamma)$. To construct the map $\iota_A$, observe that we can further decompose $([-1,1]\times F, \{0\}\times F)$ along $A'$ to obtain $([-1,1]\times F', \{0\}\times F')$ and this is a well-groomed decomposition so Case 1 applies. Hence, we have the following diagram:
\[
\xymatrix{
SFH(-M_1,-\gamma_1)\otimes SFH(-[-1,1]\times F, -\{0\}\times F)&&\\
SFH(-M_1,-\gamma_1)\otimes SFH(-[-1,1]\times F', -\{0\}\times F')\ar[u]^{{\rm id}\otimes\tilde{\iota}_A}\ar[r]^{\quad\quad\quad\quad\quad\quad\tilde{\iota}_A}&SFH(-M'',-\gamma'')\ar[r]^{\iota_{A'}}&SFH(-M,-\gamma)
}
\]
Thus, we define $\iota_A = \iota_{A'}\circ \tilde{\iota_A}\circ({\rm id}\otimes \tilde{\iota}_A)^{-1}$. It is straightforward to check that $\iota_A$ preserves contact elements. Note that since we reduce to Case 1 and 2, the map $\iota_A$ is an isomorphism. 

{\bf Case 4}. Both components of $\partial A$ are separating on $R_{\pm}(\gamma)$. Write $\al_{\pm} = \partial A\cap R_{\pm}(\gamma)$. Observe that the decomposition along $A$ and $-A$ yields exactly the same sutured manifold $(M',\gamma')$. We choose an orientation of $A$ so that $-\alpha_+$ on $R_{+}(\gamma)$ bounds a subsurface $F_+\subset R_{+}(\gamma)$ such that $F_+$ has a disconnected boundary (or equivalently, $\partial F_+$ involves some components of $\gamma$). Note that $-\alpha_-$ bounds a subsurface $F_-\subset R_-(\gamma)$. Let $S = F_+\cup A\cup F_-$. $S$ is a well-groomed surface in $(M,\gamma)$. Let
\[
(M,\gamma)\stackrel{S}{\leadsto}(M_1,\gamma_1)
\]
Note that $(M_1,\gamma_1)$ is obtained from $(M',\gamma')$ by gluing a product piece $[-1,1]\times F_{\pm}$ along $\alpha_{\pm}$. Equivalently, we have a pair of annuli $A_{\pm}$ on $(M_1,\gamma_1)$ with the following decomposition
\[
(M_1,\gamma_1)\stackrel{A_+\sqcup A_-}{\leadsto}(M',\gamma')\sqcup ([-1,1]\times F_+, \{0\}\times F_+)\sqcup ([-1,1]\times F_-, \{0\}\times F_-)
\]
See Figure \ref{fig: decomposition along A and S} for a sketch. Hence Proposition \ref{prop: surface decomp and HKM's gluing} applies to $S$ and Case 2 or 3 applies to $A_{\pm}$. 

\begin{figure}[h!]
	\begin{overpic}[width = \textwidth]{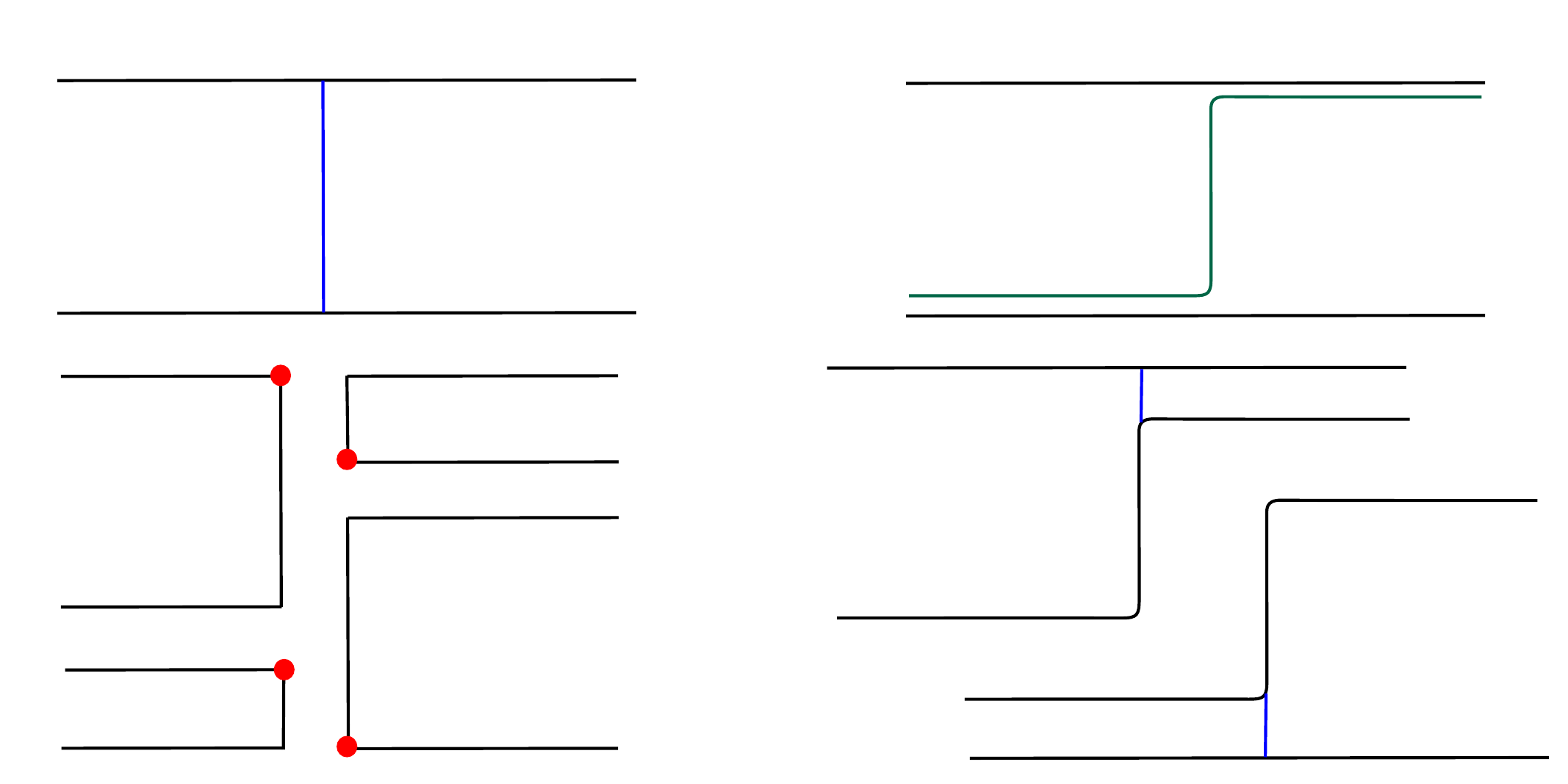}
		\put(17,47){$(M,\gamma)$}
		\put(74,47){$(M,\gamma)$}
		\put(21,37){$A$}
		\put(75,37){$S$}
		\put(30,42.5){$F_+$}
		\put(10,31){$F_-$}
		
		\put(57,18){$(M_1,\gamma_1)$}
		\put(85,9){$(M_1,\gamma_1)$}
		\put(69,24.5){$A_+$}
		\put(81,3){$A_-$}
		
		\put(7,18){$(M',\gamma')$}
		\put(27,9){$(M',\gamma')$}
		\put(24,22.5){$[-1,1]\times F_+$}
		\put(4.5,4){$[-1,1]\times F_-$}
		\put(81,3){$A_-$}
	\end{overpic}
	\caption{Top left: The original annulus $A$ inside $(M,\gamma)$, drawn in blue. Top right: The new well-groomed surface $S$, drawn in green. Bottom left: The result of decomposing along $A_{+}\sqcup A_{-}$. The red dots are the sutures. Bottom right: The new annuli $A_{\pm}$, drawn in blue.}\label{fig: decomposition along A and S}
\end{figure}

{\bf Case 5}. Exactly one component of $\partial A$ is separating on $R_{\pm}(\gamma)$. The argument in Case 4 applies to this case as well.
\epf

The following fact will be useful in later sections.
\begin{lem}\label{lem: solid torus with four longitudes }
Suppose $V$ is a solid torus and $\gamma^{4}$ consists of four longitudes of $V$. Then  $SFH(-V,-\gamma^4)\cong\mathbb{F}^2$ is generated by two contact elements.
\end{lem}

\bpf
From \cite[Section 9]{juhasz2010polytope} we know that $SFH(-V,-\gamma^4)\cong\mathbb{F}^2$. It remains to find two tight contact structures whose contact elements generate $SFH(-V,-\gamma^4)\cong\mathbb{F}^2$.  Note that we can take an oriented meridian disk $D\subset V$ that intersects $s(\gamma^4)$ at four points. The decomposition along $D$ and $-D$ both yield $D^3$ with a connected suture, which supports a unique tight contact structure. The proof of Proposition \ref{prop: product decomposition} in \cite{juhasz2008floer} ensures that the two decompositions correspond to different summands of $SFH(-V,-\gamma^4)$, and Proposition \ref{prop: surface decomp and HKM's gluing} then yields two different contact structures on $(V,\gamma^4)$, whose contact elements generate the corresponding summands.
\epf

\subsection{Sutured knot complements}\label{sec: sutured knot complement}

We will assume all knots in the paper to be (integrally) null-homologous.

\blem\label{lem: tb and suture}
Let $\xi$ be a contact structure on $Y$, and $K$ be a Legendrian knot in $Y$. Suppose $\tb(K) = n$, and $N(K)$ is a standard contact neighborhood of $K$. Perturb $\partial (Y\backslash N(K))$ to make it a convex surface with respect to $\xi|_{Y\backslash N(K)}$. Then the dividing curve on $\partial(Y\backslash N(K))$ consists of two simple closed curves on $\partial(Y\backslash N(K))$ of slope $-n$.
\elem

\bpf
After we put the standard tight contact structure on $N(K)$, the Seifert framing agrees with the torus framing given by the longitude, thus the twisting number with respect to the torus framing is the same as the $\tb(K)=n$, which implies the two dividing curves on the boundary of $N(K)$ have slope $n$. Since we are using the standard map to glue back in $N(K)$, $\partial N(K)=-\partial(Y\backslash N(K))$ also has a dividing slope $n$, which means $\partial(Y\backslash N(K))$ has a dividing slope $-n$.
\epf

Following the notion of \cite{etnyre2017sutured}, we define the suture $\Gamma_n$ to be the disjoint union of a pair of oppositely oriented simple closed curves of slope $-n$, with respect to the Seifert framing of  knots. Also, let $\Gamma_{\infty}$ denote the union of two meridians. Let $S\subset Y\backslash N(K)$ be a Seifert surface of $K$; then one can derive a grading on $SFH(Y\backslash N(K),\Gamma_{n})$ by taking one half of the evaluation of the first Chern class of relative spin${}^c$ structures on $Y\backslash N(K)$ with some fixed boundary trivialization and the fundamental class of $S$. Write this grading as $SFH(Y\backslash N(K),\Gamma_{n},i)$. We can normalize the grading so that the following holds.
\begin{itemize}
	\item We have $i\in\intg$ for $n$ odd or $n=\infty$, and $i\in\intg+\frac{1}{2}$ when $n$ even.
	\item For $|i|>g(K)+\frac{|n|-1}{2}$ with $n\in\intg$ or $|i|>g(K)$ with $n=\infty$, we have
	\begin{equation}\label{eq: adjunction}
		SFH(Y\backslash N(K),\Gamma_{n},i) = 0
	\end{equation}
	\item  For $|i|=g(K) + \frac{|n|-1}{2}$ with $n\in\intg$ or $|i|=g(K)$ with $n=\infty$, we have
	\[
	SFH(Y\backslash N(K),\Gamma_{n},i) \neq 0
	\]
\end{itemize}
Note that for any $n$, the decompositions of $(Y\backslash N(K),\Gamma_n)$ along $S$ and $-S$ both yield taut sutured manifolds. Hence the non-vanishing fact follows from \cite[Theorem 1.3 and Theorem 1.4]{juhasz2008floer}. We then choose the normalization such that these two gradings are symmetric. The vanishing fact follows from the adjunction inequality \cite[Theorem 4.1]{juhasz2010polytope} and the spin${}^c$ conjugacy.

On the knot complement $Y\backslash N(K)$ with $\Gamma_n$ sutures, we have two bypass attachments. There are two different descriptions of bypass attachments, originally introduced by Honda \cite{honda2000classification}, which we summarize as follows.
\begin{itemize}
	\item \cite{etnyre2017sutured} follows the original description from Honda, which we will also use in the current paper. In particular, the pair of bypasses corresponds to the two minimal twisting tight contact structures on $(S^1\times S^1\times [0,1],\gamma_{st})$, where $\gamma_{st}$ consists of two parallel copies of $\{pt\}\times S^1\times\{0\}$ and two parallel copies of $S^1\times \{pt\}\times \{1\}$. \cite{etnyre2017sutured} also provides a proof of the bypass exact triangle which was originally introduced in \cite{honda2000bypass}.
	\item Ozbagci \cite{ozbagci2011contact} interpret a bypass as the composition of a contact 1-handle and a contact 2-handle. \cite{juhasz1803contact} further construct maps for $SFH$ associated to contact handle attachments. \cite[Section 5]{baldwin2016contact} also used this description to construct corresponding contact handle attaching maps in monopole and instanton setups.
\end{itemize}
Note that these two bypass attachments will change the suture from $\Gamma_n$ to $\Gamma_{n+1}$, and these coincide with the positive stabilization and negative stabilization on the Legendrian knot $K$, respectively. \cite{honda2000bypass} established an exact triangle, which is re-proved in \cite[Theorem 6.1]{etnyre2017sutured}. The triangle applied to the suture $\Gamma_n$ reads as follows. Note that we have two bypasses that can be attached to $\Gamma_n$ and hence have two triangles.

\[
\xymatrix{
SFH(-(Y\backslash N(K)),-\Gamma_n)\ar[rr]^{\psi^{n}_{+,n+1}}&&SFH(-(Y\backslash N(K)),-\Gamma_{n+1})\ar[d]^{\psi^{n+1}_{+,\infty}}\\
&&SFH(-(Y\backslash N(K)),-\Gamma_{\infty})\ar[ull]^{\psi^{\infty}_{+,n}}
}
\]
\[
\xymatrix{
SFH(-(Y\backslash N(K)),-\Gamma_n)\ar[rr]^{\psi^{n}_{-,n+1}}&&SFH(-(Y\backslash N(K)),-\Gamma_{n+1})\ar[d]^{\psi^{n+1}_{-,\infty}}\\
&&SFH(-(Y\backslash N(K)),-\Gamma_{\infty})\ar[ull]^{\psi^{\infty}_{-,n}}
}
\]

The change of the relative first Chern class of a bypass attachment can be understood explicitly, and hence the maps $\psi_{\pm}$ has a degree and the above triangles have graded versions. The degrees of the maps $\psi^{n}_{\pm,n+1}$ have been computed to be $\mp\frac{1}{2}$ as in \cite[Section 3.3]{etnyre2017sutured}. The degrees of other maps in the triangles can be understood similarly. (Or one could refer to \cite[Section 4]{li2019direct} for the computation of the degrees of these maps in the monopole setup. The same argument can be applied to the Heegaard Floer setup as well.) We record the graded version of the triangles in the following lemma for future references.
\blem\label{lem: bypass exact triangle}
For a knot $K\subset Y$ of genus $g$, we have the following two exact triangles.
\[
\xymatrix{
SFH(-(Y\backslash N(K)),-\Gamma_n,i+\frac{1}{2})\ar[rr]^{\psi^{n}_{+,n+1}}&&SFH(-(Y\backslash N(K)),-\Gamma_{n+1},i)\ar[d]^{\psi^n_{+,\infty}}\\
&&SFH(-(Y\backslash N(K)),-\Gamma_{\infty},i-\frac{n}{2})\ar[ull]^{\psi^{\infty}_{+,n}}
}
\]
\[
\xymatrix{
SFH(-(Y\backslash N(K)),-\Gamma_n,i-\frac{1}{2})\ar[rr]^{\psi^{n}_{-,n+1}}&&SFH(-(Y\backslash N(K)),-\Gamma_{n+1},i)\ar[d]^{\psi^n_{-,\infty}}\\
&&SFH(-(Y\backslash N(K)),-\Gamma_{\infty},i+\frac{n}{2})\ar[ull]^{\psi^{\infty}_{-,n}}
}
\]
\elem

A technical result that will be useful later is the following.
\blem\label{lem: psi^0_1 is surjective on top grading}
Suppose $K\subset Y$ is a (null-homologous) knot that does not bounds a disk. Then the map
\[
\psi^{0}_{-,1}: SFH (-Y\backslash N(K),-\Gamma_0,g(K)-\frac{1}{2}) \to SFH (-Y\backslash N(K),-\Gamma_1,g(K))
\]
is surjective.
\elem

\bpf
Let $S\subset Y$ be a minimal-genus Seifert surface of $K$. We have a taut balanced sutured manifold decomposition
\begin{equation}\label{eq: decomposition of Gamma_0}
	(Y\backslash N(K),\Gamma_0)\stackrel{S}{\leadsto}(Y\backslash N(S),\gamma^3),
\end{equation}
where $s(\gamma^3)$ consists of three copies of $\{0\}\times \partial F \subset \partial N(S)$. Note that, other than $\Gamma_0$, the decomposition is different. More explicitly, for $n\in(\intg-\{0\})\cup\{\infty\}$, we have a sutured manifold decomposition
\[
(Y\backslash N(K),\Gamma_n)\stackrel{S}{\leadsto}(Y\backslash N(S),\gamma^1),
\]
where $s(\gamma^1) = \{0\}\times \partial F \subset \partial N(S)$. \cite[Theorem 1.3]{juhasz2008floer} then implies the following.
\begin{equation} \label{eq: Gamma_infty top grading}
    SFH(-(Y\backslash N(K)),-\Gamma_{\infty}),g(K))\cong SFH(-(Y\backslash N(S)),-\gamma^1)
\end{equation}
and 
\begin{equation}\label{eq: top grading}
	SFH(-(Y\backslash N(K)),-\Gamma_n,g(K)+\frac{n-1}{2})\cong\left\{
\begin{matrix}
	SFH(-(Y\backslash N(S)),-\gamma^3)&n=0\\
	&\\
	SFH(-(Y\backslash N(S)),-\gamma^1)&\text{$n\neq 0$ and $n\neq \infty$}\\
\end{matrix}
\right. 
\end{equation}

To relate these two Floer homology groups, when $K$ does not bound disk that is when the boundary of the surface complement is not spherical with one suture \cite[Proposition 9.2]{juhasz2010polytope} implies that 
\begin{equation}\label{eq: double dimension}
	SFH(-(Y\backslash N(S)),-\gamma^3)\cong SFH(-(Y\backslash N(S)),-\gamma^1)\otimes \mathbb{F}^2.
\end{equation} (Note in the original paper this trivial exception was not excluded in the statement of the proposition) 

Applying Lemma \ref{lem: bypass exact triangle} with $n=0$ and $i=g(K)$, we have a triangle
\[
\xymatrix{
SFH(-(Y\backslash N(K)),-\Gamma_0,g(K)-\frac{1}{2})\ar[rr]^{\psi^{0}_{-,1}}&&SFH(-(Y\backslash N(K)),-\Gamma_{1},g(K))\ar[d]^{\psi^1_{-,\infty}}\\
&&SFH(-(Y\backslash N(K)),-\Gamma_{\infty},g(K))\ar[ull]^{\psi^{\infty}_{-,1}}
}.
\]
Equations (\ref{eq: Gamma_infty top grading})(\ref{eq: top grading}) and (\ref{eq: double dimension}) then imply that $\psi^0_{-,1}$ is surjective on $SFH(-(Y\backslash N(K)),-\Gamma_0,g(K)-\frac{1}{2})$.
\epf

In \cite{etnyre2017sutured}, Etnyre, Vela-Vick, and Zarev constructed the following direct system:
\begin{equation}\label{eq: direct limit}
\xymatrix{
\cdots\ar[r]&SFH(-(Y\backslash N(K)),-\Gamma_n)[\frac{n-1}{2}]\ar[r]^{\psi_{-,n+1}^n}&SFH(-(Y\backslash N(K)),-\Gamma_{n+1})[\frac{n}{2}]\ar[r]^{\quad\quad\quad\quad\quad\quad\psi^{n+1}_{-,n+2}}&\cdots\\
}
\end{equation}
It can be extended to the following commutative diagram:
\begin{equation}\label{eq: U map}
\xymatrix{
\ar[r]&SFH(-(Y\backslash N(K)),-\Gamma_n)[\frac{n-1}{2}]\ar[r]^{\psi_{-,n+1}^n}\ar[d]^{\psi^{n}_{+,n+1}}&SFH(-(Y\backslash N(K)),-\Gamma_{n+1})[\frac{n}{2}]\ar[r]^{\quad\quad\quad\quad\quad\quad\psi^{n+1}_{-,n+2}}\ar[d]^{\psi^{n+1}_{+,n+2}}&\\
\ar[r]&SFH(-(Y\backslash N(K)),-\Gamma_{n+1})[\frac{n}{2}]\ar[r]^{\psi_{-,n+2}^{n+1}}&SFH(-(Y\backslash N(K)),-\Gamma_{n+2})[\frac{n+1}{2}]\ar[r]^{\quad\quad\quad\quad\quad\quad\quad\psi^{n+2}_{-,n+3}}&\\
}
\end{equation}

For any suture $\Gamma_{n}$ for $n\in\intg$, one can attach a contact $2$-handle along a meridian of the knot. The resulting sutured manifold is $Y(1)=(Y-B^3,\delta)$, where $s(\delta)$ is a connected simple closed curve on $\partial B^3$. As in \cite[Proposition 9.1]{juhasz2006holomorphic}, we know that $SFH(Y(1)) = \widehat{HF}(Y)$. The functoriality of contact gluing map in \cite[Proposition 6.2]{honda2008contact} yields the following commutative diagram
\begin{equation}\label{eq: pi map}
\xymatrix{
\cdots\ar[r]&SFH(-(Y\backslash N(K)),-\Gamma_n)[\frac{n-1}{2}]\ar[r]^{\psi_{-,n+1}^n}\ar[d]^{\pi^n}&SFH(-(Y\backslash N(K)),-\Gamma_{n+1})[\frac{n}{2}]\ar[r]^{\quad\quad\quad\quad\quad\quad\psi^{n+1}_{-,n+2}}\ar[dl]^{\pi^{n+1}}&\cdots\\
&SFH(-Y(1))=\widehat{HF}(-Y)&&
}
\end{equation}

\bthm[{\cite[Theorem 1.1 and Theorem 1.3]{etnyre2017sutured}}]\label{thm: direct limit = HFK^-}
The direct limit of (\ref{eq: direct limit}) is isomorphic to $HFK^-(-Y,{K})$ via a grading-preserving isomorphism $\mathcal{I}$. Furthermore, under this isomorphism, we have the following.
\begin{itemize}
	\item The collection of maps $\{\psi^{n}_{+,n+1}\}$ in (\ref{eq: U map}) induces a map on the direct limit that coincides with the $U$ action on $HFK^-(-Y,{K})$.
	\item The collection of maps $\{\pi^n\}$ in (\ref{eq: pi map}) induces a map on the direct limit that coincides with the map $\pi_*$ which is the map induced on homology by setting the formal variable U equal to one at the chain level, and fits into the following exact triangle.
	\begin{equation}\label{eq: cone of U-1}
		\xymatrix{
	HFK^-(-Y,K)\ar[rr]^{U-1}&&HFK^-(-Y,K)\ar[dl]^{\pi_*}\\
	&\widehat{HF}(-Y)\ar[ul]&
	}
	\end{equation}
\end{itemize}
\ethm

\subsection{The $\tau$ invariant in Heegaard Floer theory}\label{subsec: tau in HF}
The $\tau$ invariant in Heegaard Floer theory was first introduced by Ozsv\'ath and Szab\'o in \cite{ozsvath2003tau}. It defines concordance homomorphism
\[
	\tau: \mathcal{C}\to\mathbb{Z}
\]
where $\mathcal{C}$ denotes the smooth cocordance group of knots in $S^3$. There is another interpretation from \cite[Theorem 1.5]{ozsvath2008legendrian} and \cite[Section 2.5]{manolescu2014lecturenotes} that will be used in the current paper. For a knot $K\subset S^3$, we have
\[
\tau(K) = \max\{i~|~\exists x\in {HFK}^-(-S^3,K,i)~s.t.~\pi_{*}(x) \neq 0\}.
\]
Here $\pi_*$ is the map as in (\ref{eq: cone of U-1}) for $Y=S^3$. Note in the literature, there might be a negative sign in the definition but we work with $K\subset -S^3$ which is equivalent to the mirror of $K$ in $S^3$. Also, the definition might involve the free tower of $HFK^-$, but the elements in the free tower are exactly the ones that survive under the map $\pi_*$.

We have the following lemma. 
\blem\label{lem: pi^0 is non-trivial on top grading}
If $K\subset S^3$ is not the unknot, and $\tau(K) = g(K)$. Then the map
\[
\pi^0: SFH(-S^3\backslash N(K),-\Gamma_0,g(K)-\frac{1}{2}) \to \widehat{HF}(-S^3)
\]
is non-trivial.
\elem

\bpf
It suffices to show that the map
\[
\pi^1: SFH(-S^3\backslash N(K),-\Gamma_1,g(K)) \to \widehat{HF}(-S^3)
\]
is non-trivial. Then it is clear that the lemma follows from Equation \ref{eq: pi map} and Lemma \ref{lem: psi^0_1 is surjective on top grading}.

To show that $\pi^1$ is non-trivial on the top grading, we applying Lemma \ref{lem: bypass exact triangle} with $n\geq 1$ and $i=g(K)+\frac{n}{2}$, and we derive a triangle
\[
\xymatrix{
SFH(-(S^3\backslash N(K)),-\Gamma_n,g(K)+\frac{n-1}{2})\ar[rr]^{\psi^{n}_{-,n+1}}&&SFH(-(S^3\backslash N(K)),-\Gamma_{n+1},g(K)+\frac{n}{2})\ar[d]^{\psi^{n+1}_{-,\infty}}\\
&&SFH(-(S^3\backslash N(K)),-\Gamma_{\infty},g(K)+n)\ar[ull]^{\psi^{\infty}_{-,n}}
}.
\]
The adjunction inequality in (\ref{eq: adjunction}) then implies that the map $\psi^n_{-,n+1}$ is an isomorphism on $SFH(-(S^3\backslash N(K)),-\Gamma_n,g(K)+\frac{n-1}{2})$ for any $n\geq 1$. Hence the direct system (\ref{eq: direct limit}) stabilizes on the top grading, and we have a commutative diagram according to Theorem \ref{thm: direct limit = HFK^-}
\begin{equation}\label{eq: pi_1 and pi_star}
	\xymatrix{
SFH(-(S^3\backslash N(K)),-\Gamma_1,g(K))\ar[rr]^{\quad\quad\quad\mathcal{I}}_{\quad\quad\quad\cong}\ar[dr]^{\pi^1}&&HFK^-(-S^3,K,g(K))\ar[dl]^{\pi_*}\\
&\widehat{HF}(-S^3)&
}
\end{equation}
The condition $\tau(K)= g(K)$ implies that $\pi_*$ is non-trivial in top grading. Hence, we are done.
\epf

\section{Constructing $\rm tb=0$ representatives} \label{sec: constructing tb}
In this section, we construct $\tb=0$ representatives for various families of knots.

\subsection{Nearly fibered knot with $\tau = g$}
In this subsection, we aim to prove Theorem \ref{thm: nearly fibered knot}.

\blem\label{lem: top grading of a nearly fibered knot is generated by contact elements}
Let $K\subset S^3$ be a nearly fibered knot, {\it i.e.}
	\[
	\widehat{HFK}(S^3,K,g(K)) \cong \mathbb{F}^2.
	\]
Then $SFH(-S^3\backslash N(K), -\Gamma_0,g(K)-\frac{1}{2})$ is generated by contact elements.
\elem

\bpf
Let $S$ be a minimal genus Seifert surface of $K$. As in Equation (\ref{eq: decomposition of Gamma_0}), we know that there is a well-groomed decomposition
\[
(S^3\backslash N(K),\Gamma_0)\stackrel{S}{\leadsto} (S^3\backslash N(S),\gamma^3)
\]
where $\gamma^3$ consists of three parallel copies of $\partial S$. Furthermore, according to Proposition \ref{prop: Juh well-groomed sfh map} there is an injective map
\[
\iota_S: SFH(-S^3\backslash N(S),-\gamma^3) \to SFH(-S^3\backslash N(K),-\Gamma_0)
\]
which is an isomorphism. As a result, by Proposition \ref{prop: surface decomp and HKM's gluing}, it suffices to show that $SFH(-S^3\backslash N(S),-\gamma^3)$ is generated by contact elements.

We now regard $N(S)$ as $[-2,2]\times \partial S$ and $\gamma^3 = \{-1,0,1\}\times \partial S$. We can now take an annulus $A_0 = [-2,2]\times \partial S$ and push it out to obtain a product annulus inside $(S^3\backslash N(S),\gamma^3)$. A decomposition yields
\[
(S^3\backslash N(S),\gamma^3)\stackrel{A_0}{\leadsto}(V,\gamma^4) \sqcup (S^3\backslash N(S),\gamma^1 = \{0\}\times \partial S).
\]

By \cite[Proposition 9.2]{juhasz2010polytope}, we have
\[
SFH(-S^3\backslash N(S),-\gamma^3) \cong SFH(-S^3\backslash N(S),-\gamma^1)\otimes \mathbb{F}^2.
\]
Since we already know that $SFH(-V,-\gamma^4)\cong \mathbb{F}^2$ and is generated by contact elements, as in Lemma \ref{lem: solid torus with four longitudes }, the map $\iota_{A_0}$ from Proposition \ref{prop: product annulus decomposition and HKM's gluing} is an isomorphism and it suffices to show that $SFH(-S^3\backslash N(S),-\gamma^1)$ is generated by contact elements. 

According to \cite{LY2023nearly}, we know that there is a collection of pair-wise disjoint non-trivial product annuli $\mathcal{A}\subset (S^3\backslash N(S),\gamma^1)$ such that there is a decomposition
\[
(S^3\backslash N(S),\gamma^1)\stackrel{\mathcal{A}}{\leadsto} (M,\gamma) \sqcup ([-1,1]\times F,\{0\}\times \partial F).
\]
Here $F$ is a compact oriented surface so that each component of $F$ has non-empty boundary, and $(M,\gamma)$ is in one of the following three cases.
\begin{itemize}
	\item [(i)] $(M,\gamma) = (V,\gamma^4)$.
	\item [(ii)] $M = V$ and $\gamma$ consists of two curves of slope $2$. 
	\item [(iii)] $M=S^3\backslash N(T)$ for the right handed trefoil $T$ and $\gamma$ consists of two curves of slope $2$. 
\end{itemize}
It is straightforward to compute that in any of the three cases
\[
SFH(-M,-\gamma) \cong \mathbb{F}^2
\]
so the map $\iota_{\mathcal{A}}$ from Proposition \ref{prop: product annulus decomposition and HKM's gluing} is an isomorphism. Again,  it has been known that $SFH(-V,-\gamma^4)$ is generated by contact elements, and the fact that in case (ii) $SFH(-M,-\gamma)$ is generated by contact elements follows from a similar argument for case (i) (as in the proof of Theorem \ref{thm: Tight contact strucutre from fibered tau}): we can take a meridian disk $D$ of $M=V$ and look at the decomposition along $D$ and $-D$. For case (iii), instead of a meridian disk, we can take a genus-one Seifert surface $S$ of the trefoil knot $T$. Both $S$ and $-S$ are well-groomed surfaces and the decompositions of $(M,\gamma )$ along them both yield a product sutured manifold. As a result we can apply Proposition \ref{prop: surface decomp and HKM's gluing} to obtain two tight contact structures on $(M,\gamma)$ whose contact elements generate $SFH(-M,-\gamma)$. Thus we complete the proof.
\epf
Now we are ready to prove the Theorem \ref{thm: nearly fibered knot}, and we restate it below.
\begin{thm}\label{thm: nearly fibered knots}
	 Let $K\subset S^3$ be a nearly fibered knot, {\it i.e.}
	\[
	\widehat{HFK}(S^3,K,g(K)) \cong \mathbb{F}^2.
	\]
	Suppose  $\tau(K) = g(K)$, then $K$ has a Legendrian representative $L_K$ such that  $tb(L_K)=0$ and realizes the Thurston--Bennequin bound, i.e. $tb(L_K)-rot(L_K) = 2g(K) - 1$.
\end{thm}

\bpf
Because $\tau(K)=g(K)$, by Lemma \ref{lem: top grading of a nearly fibered knot is generated by contact elements} and Lemma \ref{lem: pi^0 is non-trivial on top grading} we can assume that there exists a tight contact structure $\xi_K$ on $(S^3\backslash N(K),\Gamma_0)$ such that 
\[
	\pi^0(c(\xi_K))\neq 0.
\]
Let $\xi$ be the result of attaching a contact $2$-handle along a meridian of the knot $K$ to $\xi_K$ on $(S^3 \backslash N(K),\Gamma_0)$. Note that $\xi$ is on the sutured manifold $S^3(1)$ obtained by digging out a $3$-ball from $S^3$ and taking a connected curve as the suture. Then the fact that a contact handle attaching map preserving contact element implies that
\[
c(\xi) = \pi^0(c(\xi_K))\neq0.
\]
However, there is a unique tight contact structure on $S^3(1)$, which is the standard one. So $\xi$ corresponds to the standard tight contact structure $\xi_{std}$ on $S^3$.

Moreover, we note that the standard Legendrian neighborhood of a Legendrian knot can be decomposed into a contact $0$-handle and a contact $1$-handle. Gluing it to the knot complement flips the handles and makes them a contact $3$-handle and a contact $2$-handle. To obtain $S^3$, the $2$-handle must be glued along the meridian of $K$, and this gluing coincides with the construction of $\pi^0$. Therefore, we conclude that the knot $L_K$ as the core of the neighborhood is Legendrian with respect to $\xi_{std}$ on $S^3$ since the slope of dividing set is $0$, Lemma \ref{lem: tb and suture} implies that $\tb(L_K)=0$. 

To show that $L_K$ realizes the three-dimensional Thurston-Bennequin bound we use the grading of the contact element $c(\xi_K)$. By the definition of the $Eh$ invariant, we know that $c(\xi_K)=Eh(L_k)$ is non-trivial in the top grading. Moreover, according to  \cite{etnyre2017sutured} we know the limit map of the direct system will send $Eh(L_k)$ to the LOSS invariant (\cite{LOSSInvariant}) $\mathcal{L}(L_k)$ of $L_k$. Thus, $c(\xi)$ is non-vanishing together with the map $\pi_*$ in \ref{eq: pi_1 and pi_star} being non-trivial on the top grading tells us that $\mathcal{L}(L_k)$ is non-trivial in grading $g(K)$. Last, \cite[Corollary 1.7]{OzsvathStipsiczContactsurgeryandtransverseinvariant} says that $tb(L_k)-rot(L_k)=2g(K)-1$,  which finishes the proof.  
\epf

\subsection{Knots realizing Thurston-Bennequin bound} \label{subsec: Knots realizing Thurston-Bennequin bound}
In this section, we work with knots $K\subset S^3$ that admit a Legendrian representative $L$ such that
\[
tb(L) - rot(L) = 2g(K) - 1.
\]

\blem\label{lem: contact element on top grading}
Suppose $K$ has a Legendrian representative $L$ realizing the three-dimensional Thurston-Bennequin bound. Assume without loss of generality that $L$ has $tb = -n$ for a large positive integer $n$ (we can always do this by negative stabilizing the Legendrian). Let $N(L)$ be a standard contact neighborhood of $L$, and $\xi_L$ be the restriction of $\xi_{std}$ on $S^3\backslash N(L)$. Then
\[
c(\xi_L) \in SFH(-S^3\backslash N(K),-\Gamma_n,g(K)+\frac{n-1}{2}).
\]
Furthermore, we have
\[
\pi^n(c(\xi_L)) = c(\xi_{std}) \in SFH (-S^3(1)).
\]
\elem

\bpf
As we just saw at the end of the proof of \ref{thm: nearly fibered knots}, $c(\xi_L)=EH(L,\xi_L)$ and the limit of $EH(L,\xi_L)$ under the direct system \ref{eq: direct limit} is the same as the LOSS invariant $\mathcal{L}(L)$  of $L$. Moreover, by \cite[Corollary 1.7]{OzsvathStipsiczContactsurgeryandtransverseinvariant}, the condition $tb(L)-rot(K)=2g(K)-1$ implies that the Alexander grading of $\mathcal{L}(L)$ is $g(K)$ . Since under the direct system $(-S^3\backslash N(K),-\Gamma_n,g(K)+\frac{n-1}{2})$ maps to $HFK^-(-Y,K,g(K))$,  $EH(L,\xi_L)$ maps to $\mathcal{L}(L)$, and  $\mathcal{L}(L)\in HFK^-(-Y.K.g(K))$, we must have $EH(L,\xi_L)=c(\xi_L) \in SFH(-S^3\backslash N(K),-\Gamma_n,g(K)+\frac{n-1}{2})$.

For the second part, it is more straightforward because the map $\pi^n$ corresponds to contact $2-$handle attachment, and it will sends $c(\xi_L)$ to $c(\xi_{std})$ as $\xi_L$ is the restriction of $\xi_{std}$.
\epf

The point of this Lemma is that to prove Theorem \ref{thm: tb=0 for knot realizeing T-B bound}, we want to find a contact structure $\xi_0$ on $(S^3\backslash N(K),\Gamma_0)$ such that $c(\xi_0) \in SFH(-S^3\backslash N(K),-\Gamma_{0},g(K)-\frac{1}{2})$ maps to $c(\xi_L) \in SFH(-S^3\backslash N(K),-\Gamma_n,g(K)+\frac{n-1}{2})$ in the  direct system \ref{eq: direct limit}. Then the commutative diagram \ref{eq: pi map} will imply that $\pi^0(\xi_0)=\pi^n(c(\xi_L)) = c(\xi_{std})$. Thus, just like the end of the proof of Theorem \ref{thm: nearly fibered knots}, we find the $L_K$ as the core of the solid torus (neighborhood of $K$) with slope $0$ dividing sets.

With the above idea in mind, we are ready to prove Theorem \ref{thm: tb=0 for knot realizeing T-B bound}, and again we restate the theorem below.

\bthm \label{thm: restated 1.1}
Suppose $K$ is a non-trivial knot in $S^3$ and has a Legendrian representative $L$ realizing the three-dimensional Thurston-Bennequin bound. Assume without loss of generality that $L$ has $tb = -n$ for a large positive integer $n$. Then, $K$ has a Legendrian representative $L_K$ such that $tb(L_K) = 0$.
\ethm
\bpf

Let $N(K)$ be a tubular neighborhood of $K$. From Lemma \ref{lem: contact element on top grading}, there exists a large enough $n$ such that there exists a contact structure $\xi_L$ on $(S^3\backslash N(K),\Gamma_n)$ such that
\[
c(\xi_L) \in SFH(-S^3\backslash N(K),-\Gamma_n,g(K)+\frac{n-1}{2})~{\rm and}~\pi^n(c(\xi_L)) = c(\xi_{std}) \in SFH (-S^3(1)).
\]

Following the idea above, we want to construct a $c(\xi_0)$ which is non-zero in $SFH(-S^3\backslash N(K),-\Gamma_{0},g(K)-\frac{1}{2})$ and it maps to $c(\xi_L)$ under the direct system.

We claim that we can pass the construction of contact structure on the knot complement to the contact structure of the Seifert surface complement. 

More precisely, for any $k\geq 1$, as in \cite{honda2000bypass}, for the sutured manifold $(S^3\backslash N(K),\Gamma_k)$, bypass attachments correspond to the $\psi^k_{-,k+1}$ are performed along an arc $\beta$ which has both end points on the suture $\Gamma_k$. We can perform suture decomposition along the (minimal-genus) Seifert surface $S$ of $K$ so that $S\cap \beta = \emptyset$. As a result, the maps associated to the bypass attachments commute with the maps come from decomposition along the surface $S$. More precisely, we have the following commutative diagram
\begin{equation}\label{eq: comm. dia. surf. decomp and bypass}
	\xymatrix{
SFH(-S^3\backslash N(K),-\Gamma_{k},g(K)+\frac{k-1}{2})\ar[rr]^{\psi^k_{-,k+1}}&&SFH(-S^3\backslash N(K),-\Gamma_{k+1},g(K)+\frac{k}{2})\\
SFH(-S^3\backslash N(S),-\gamma^{1})\ar[rr]^{\psi}\ar[u]^{\iota_{S,k}}&&SFH(-S^3\backslash N(S),-\gamma^{1})\ar[u]^{\iota_{S,k+1}}
}
\end{equation}
Here, the vertical maps $\iota_{S,k}$ and $\iota_{S,k+1}$ are the maps from Proposition \ref{prop: surface decomp and HKM's gluing}, and $\gamma^1$ follows the above notations which denotes one copy of $\partial S$ as the suture. Since the bypass arc $\beta$ is disjoint from $S$, it survives under the sutured manifold decomposition.  Hence, it induces a bypass attachment on $(S^3\backslash N(S),\gamma^1)$, resulting in $(S^3\backslash N(S),\gamma^1)$ as well. The map associated to this bypass attachment on $(S^3\backslash N(S),\gamma^1)$ is denoted by $\psi$. See Figure \ref{fig: bypass on Gamma_k}.

\begin{figure}[h!]
	\begin{overpic}[width = 0.8\textwidth]{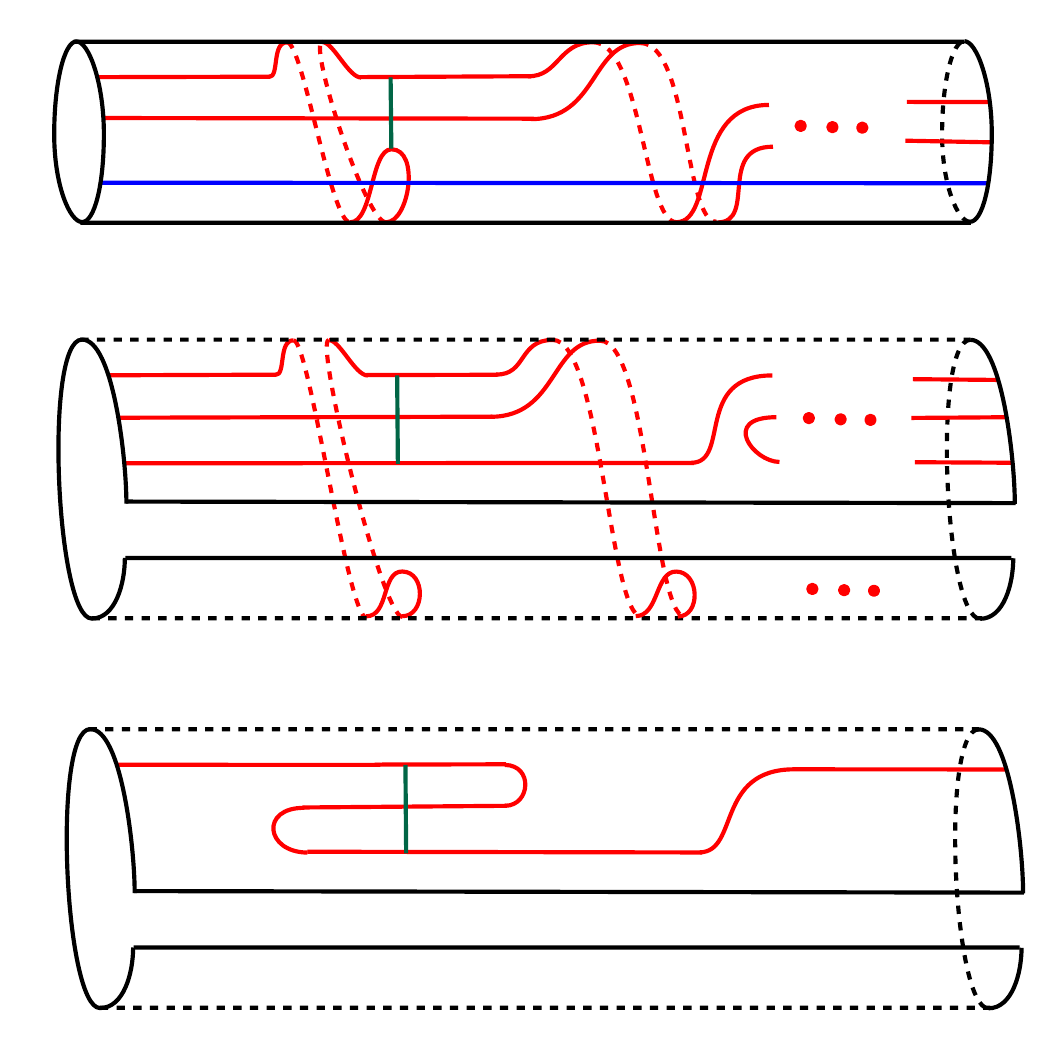}
		\put(38,90) {\color{lygreen}$\beta$}
		\put(96,87.5) {\color{red}$\Gamma_k$}
		\put(96,81.5) {\color{blue}$\partial S$}
		\put(98,57.5) {\color{red}$-\gamma^1$}
	\end{overpic}
	\caption{Top: the sutured manifold $(S^3\backslash N(K),\Gamma_{k})$ for $k\geq 1$. The suture is drawn in red, the Seifert surface $S$ is drawn in blue and the bypass arc $\beta$ is drawn in green. We isotope the suture in the figure, which is equivalent to isotoping $\partial S$ as in the proof. Middle: the sutured manifold $(-S^3\backslash N(S),-\gamma^1)$ after decomposing along $S$. Bottom: the suture after isotopy. It is then obvious that the bypass attachment is trivial.}\label{fig: bypass on Gamma_k}
\end{figure}

Now, in our situation, we decompose the contact structure $\xi_L$ on $(S^3\backslash N(K),\Gamma_n)$ along the convex surface $S$ to obtain a tight contact structure $\xi_S$ on $(S^3\backslash N(S),\gamma^1)$. Then, we can further glue $\xi_S$ along $S$ back to obtain a tight contact structure $\xi_k$ on $(S^3\backslash N(K),\Gamma_k)$ for each $k\geq 1$, where the gluing map corresponds to $\iota_{S,k}$, and Proposition \ref{prop: surface decomp and HKM's gluing} implies that 
\[
\iota_{S,k}(c(\xi_S)) = c(\xi_k)\in SFH(-S^3\backslash N(K),-\Gamma_{k},g(K)+\frac{k-1}{2})
\]
for any $k\geq 1$.

It is straightforward to check that the bypass arc $\beta$ on $(S^3\backslash N(S),\gamma^1)$ is a trivial one; as a result, we know from \cite{honda2002gluing} that the bypass attachment along $\beta$ preserves $\xi_S$, and hence
\[
\psi (c(\xi_S)) = c(\xi_{S}).
\]
From the commutativity of the Equation \ref{eq: comm. dia. surf. decomp and bypass}, we then conclude that
\[
\psi^{k}_{-,k+1}(c(\xi_k)) = c(\xi_{k+1}).
\]
From Equation (\ref{eq: pi map}), and the fact that
\[
\pi^n(c(\xi_L)) = c(\xi_{std}) \in SFH(-S^3(1))
\]
we know that
\begin{equation} \label{eq: pi^n of xi_k}
\pi^1(c(\xi_1)) = \pi^2(c(\xi_2)) = \cdots = \pi^n(c(\xi_L)) = c(\xi_{std}) \in SFH(-S^3(1)).
\end{equation}

Take $k=0$ now, we have a similar commutative diagram
\begin{equation}\label{eq: comm. dia. surf. decomp and bypass, k=0}
	\xymatrix{
SFH(-S^3\backslash N(K),-\Gamma_{0},g(K)-\frac{1}{2})\ar[rr]^{\psi^0_{-,1}}&&SFH(-S^3\backslash N(K),-\Gamma_{1},g(K))\\
SFH(-S^3\backslash N(S),-\gamma^{3})\ar[rr]^{\psi}\ar[u]^{\iota_{S,0}}&&SFH(-S^3\backslash N(S),-\gamma^{1})\ar[u]^{\iota_{S,1}}
}
\end{equation}
See Figure \ref{fig: bypass on Gamma_0}.

\begin{figure}[h!]
	\begin{overpic}[width = 0.8\textwidth]{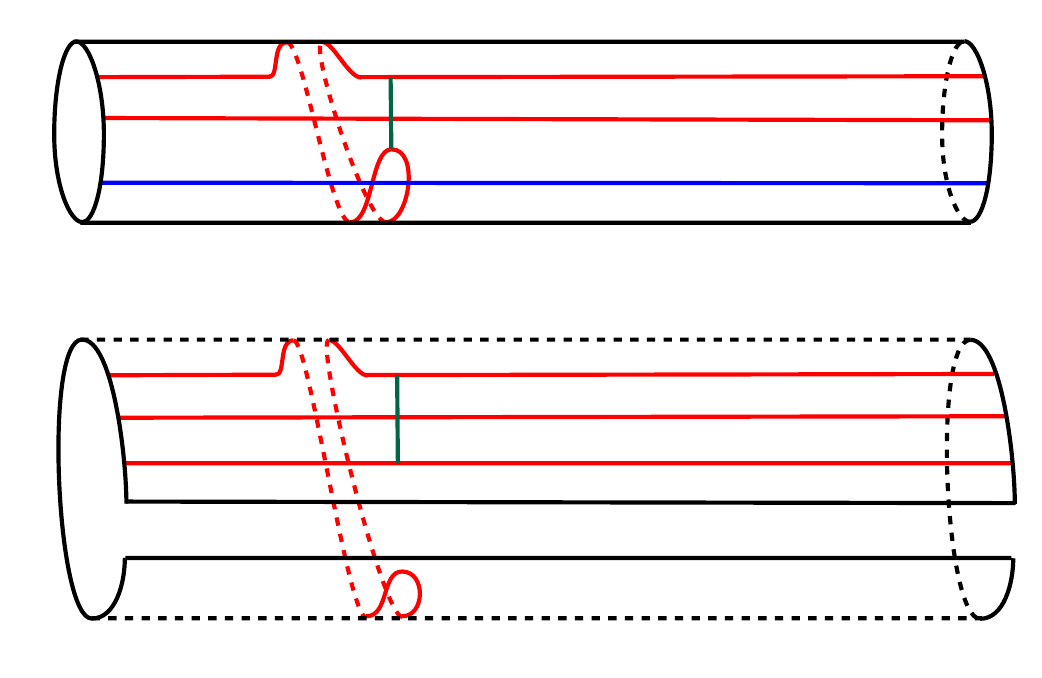}
		\put(38,55) {\color{lygreen}$\beta$}
		\put(96,55) {\color{red}$\Gamma_0$}
		\put(96,46.5) {\color{blue}$\partial S$}
		\put(98,25) {\color{red}$-\gamma^3$}
	\end{overpic}
	\caption{Top: the sutured manifold $(S^3\backslash N(K),\Gamma_{0})$. The suture is drawn in red, the Seifert surface $S$ is drawn in blue and the bypass arc $\beta$ is drawn in green. We isotope the suture in the figure, which is equivalent to isotope $\partial S$ as in the proof. Bottom: the sutured manifold $(-S^3\backslash N(S),-\gamma^3)$ after decomposing along $S$.}\label{fig: bypass on Gamma_0}
\end{figure}

Thus, the commutative diagram \ref{eq: comm. dia. surf. decomp and bypass, k=0} above together with the Equation \ref{eq: pi^n of xi_k} and the diagram \ref{eq: pi map} imply that it suffices to construct a tight contact structure $\xi_0'$ on $(S^3\backslash N(S),\gamma^{3})$ such that
\begin{equation} \label{eq: enough for c0=c1}
    \psi (c(\xi_0')) = c(\xi_S).
\end{equation}Because then the contact structure $\xi_0$ on $(S^3\backslash N(K),\Gamma_{0})$ that we want construct is obtained by gluing $\xi_0'$ along $S$ under map correspond to $\iota_{S,0}$, and we have $$\pi^0(c(\xi_0))=\pi^0(\iota_{S,0}(c(\xi_0')))=\pi^1(\psi^0_{-,1}(\iota_{S,0}(c(\xi_0'))))=\pi^1(\iota_{S,1}(\psi(c(\xi_0'))))=\pi^1(\iota_{S,1}(c(\xi_s)))=\pi^1(c(\xi_1))=c(\xi_{std}).$$

To find $\xi_0'$, we first view $N(S) = [-2,2]\times S$ and $\gamma^3 = \{-1,0,1\}\times \partial S$. We then take $A$ an annulus as a push off of $[-2,2]\times \partial S$ and again since $A$ is disjoint from the bypass arc $\beta$, we have a commutative diagram

\begin{equation}\label{eq: comm. dia. surf. decomp and bypass, k=0}
	\xymatrix{
SFH(-S^3\backslash N(S),-\gamma^{3})\ar[rr]^{\psi}&&SFH(-S^3\backslash N(S),-\gamma^{1})\\
SFH (-V,-\gamma^4)\otimes SFH(-S^3\backslash N(S),-\gamma^{1})\ar[rr]^{\psi_V\otimes id}\ar[u]^{\iota^3_{A}}&&SFH (-V,-\gamma^2)\otimes SFH(-S^3\backslash N(S),-\gamma^{1})\ar[u]^{\iota^1_{A}}
}
\end{equation}
Note that we have $\iota_A^3\neq 0$ (and actually it is an isomorphism) only if $K$ is non-trivial. $\gamma^2\subset \partial V$ consists of two copies of longitudes of the solid torus $V$, which means $(V,\gamma^2)$ is a product sutured manifold. Hence, by Corollary \ref{cor: unique contact structure on product sutured manifold} there is a standard tight contact structure $\xi^2_V$ on $(V,\gamma^2)$ such that $c(\xi^2_V)$ generates $SFH(-V,-\gamma^2)\cong\mathbb{F}$.

The bypass arc $\beta$ on $(V,\gamma^4)$ induces an exact triangle
\[
\xymatrix{
SFH(-V,-\gamma^2)\ar[rr]&&SFH(-V,-\gamma^2)\ar[dl]\\
&SFH(-V,-\gamma^4)\ar[ul]^{\psi_V}&
}
\]
See Figure \ref{fig: bypass on V}.
\begin{figure}[h!] \label{fig: bypass on V}
	\begin{overpic}[width = 0.8\textwidth]{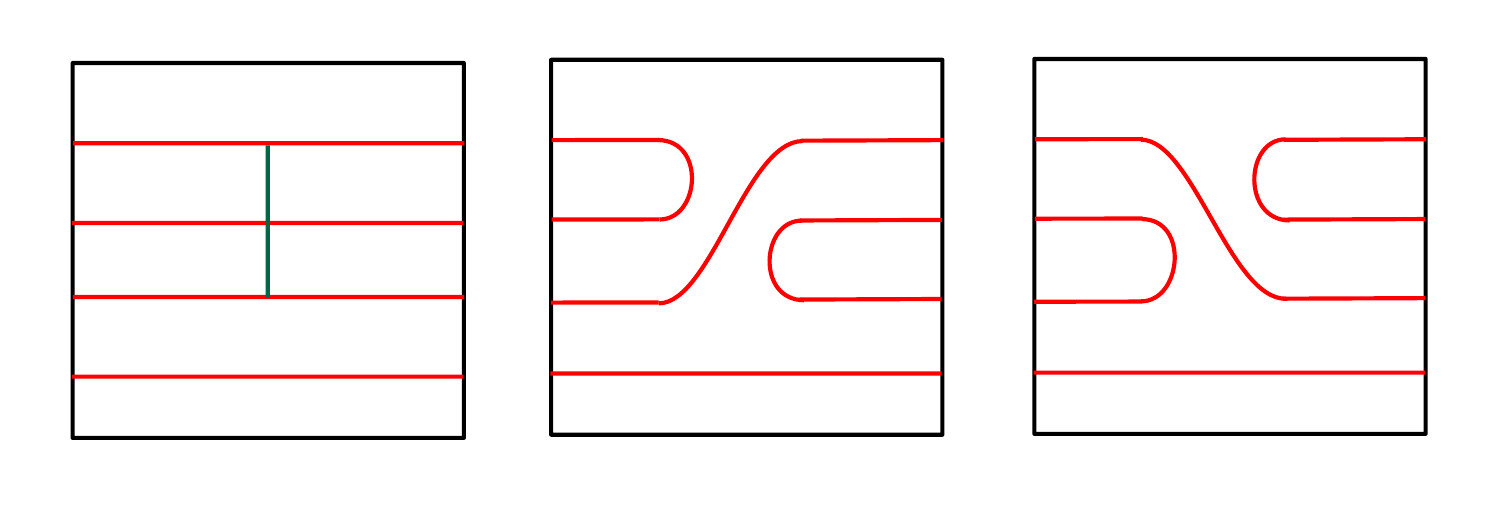}		
	\end{overpic}
	\caption{The triad associated to the bypass arc $\beta$ on $(V,\gamma^4)$. The square represents $\partial V\cong T^2$ in the standard way.}\label{fig: bypass on V}
\end{figure}

Since $SFH(-V,-\gamma^4)\cong \mathbb{F}^2$ and is generated by contact elements, and we know from the above exact triangle that $\psi_V$ is surjective. Hence, we can pick a tight contact structure $\xi_V^4$ on $(V,\gamma^4)$ such that
\[
\psi_V (c(\psi_V^4)) = c(\psi_V^2).
\]
\epf
Thus, the gluing result in Proposition \ref{prop: product annulus decomposition and HKM's gluing} gives rise to a tight contact structure $\xi_0'$ on $(S^3\backslash N(S),\gamma^{3})$ such that
\[
\psi (c(\xi_0')) = c(\xi_S).
\]

\subsection{Results in in other $3-$manifolds } \label{subsection: Fibered knots}
In this section, we prove Corollary \ref{thm: Tight contact strucutre from fibered tau} and Theorem \ref{thm: tb=0 for 2g-1 in other 3-manifold}. First, note as in Section \ref{subsec: tau in HF}, for any non-zero element $c \in \widehat{HF}(-Y)$, $\tau_{c}(K)$ can be defined as follows:
\[
\tau_c(K) = \max\{i~|~\exists x\in {HFK}^-(-Y,K,i)~s.t.~\pi_{*}(x) =c \}.
\] 

Using the generalized $\tau$ invariant, Lemma \ref{lem: pi^0 is non-trivial on top grading} can be directly generalized to $\tau_c$ in $Y$. 

\begin{lem} \label{lem: tau version pi^0 is non-trivial on top grading}

If $K\subset Y$ is non-trivial, and $\tau_c(K) = g(K)$ for some $c\in \widehat{HF}(-Y)$. Then the map
\[
\pi^0: SFH(-Y \backslash N(K),-\Gamma_0,g(K)-\frac{1}{2}) \to \widehat{HF}(-Y)
\]
is non-trivial.
\end{lem}

Now we are going to prove Theorem \ref{thm: Tight contact strucutre from fibered tau} about Legendrian fibered knots in other contact $3-$manifolds.

\bpf[Proof of Corollary \ref{thm: Tight contact strucutre from fibered tau}]
We will first  find a contact structure on $(Y\backslash N(K),\Gamma_0)$ and show after gluing back the solid torus with slope $0$, it is also tight. 

Let $K\subset Y$ be a non-trivial knot with a non-planer Seifert surface $S$. As in Equation (\ref{eq: decomposition of Gamma_0}) we know that there is a well-groomed decomposition
\[
(Y\backslash N(K))\stackrel{S}{\leadsto} (Y\backslash N(S),\gamma^3)
\]

We now regard $N(S)$ as $[-2,2]\times \partial S$ and $\gamma^3 = \{-1,0,1\}\times \partial S$. We can now take an annulus $A_0 = [-2,2]\times \partial S$ and push it out to obtain a product annulus inside $(Y\backslash N(S),\gamma^3)$. A decomposition yields
\[
(Y\backslash N(S),\gamma^3)\stackrel{A_0}{\leadsto}(V,\gamma^4) \sqcup (Y\backslash N(S),\gamma^1 = \{0\}\times \partial S).
\]

Since $K$ is fibered, we know that $(Y\backslash N(S),\gamma^1)$ is a product sutured manifold. So the proof of Lemma \ref{lem: top grading of a nearly fibered knot is generated by contact elements} and Theorem \ref{thm: nearly fibered knots} applies here as well. In particular, we can find two tight contact structures $\xi_2$ and $\xi_2'$ on $(Y\backslash N(K),\Gamma_0)$ so that $c(\xi_2)$ and $c(\xi_2')$ generates $SFH(-(Y\backslash N(K)),-\Gamma_0,g(K)-\frac{1}{2})$. 

Now, since $c(\xi_K)\neq 0$ by the definition of contact invariant in a general contact  $3-$manifold and the proof of Theorem 5 in \cite{HeddenOSinvariantofknotsincontactmanifold} given by Hedden, we know $g(K)=\tau_{c(\xi_k)}(K)$ where $c(\xi_K) \in \widehat{HF}(-Y)$ is the contact invariant of $(Y,\xi)$. Then Lemma \ref{lem: tau version pi^0 is non-trivial on top grading} applies to our current situation, and we conclude that the map
\[
\pi^0: SFH(-(Y\backslash N(K)),-\Gamma_0,g(K)-\frac{1}{2})\to \widehat{HF}(-Y)
\]
is non-trivial, we may assume, without loss of generality, that $\pi^0(\xi_2)\neq 0$. Note that $\pi^0$ is induced by a contact $2$-handle attachment along the meridian of the knot, thus we obtain a contact structure $\xi'$ on $Y$ with the property that
\[
c(\xi') = \pi^0(c(\xi_2)) \neq 0.
\]
As in the proof of Theorem \ref{thm: nearly fibered knots}, we conclude that the knot $L_K$ as the core of the neighborhood is Legendrian with respect to $\xi'$ on $Y$, and since the slope of dividing set is $0$, Lemma \ref{lem: tb and suture} implies that $\tb(L_K)=0$. 

Lastly, we want to show $\xi'=\xi_K$ up to adding Giroux torsion in the complement of $L_K$. To obtain this result, we use the Theorem 1.1 in \cite{EtnyreVanHMFiberedBennequin} by showing that the transverse push-off of the Legendrian representative $L_K$ realizes the Bennequin bound in $(Y,\xi')$. 

Again, by the definition of the $Eh$ invariant of  Legendrian knot in $(Y,\xi')$, we know $Eh(L_K)=c(Y\backslash L_K,\xi')=c(\xi_2) \in SFH(-(Y\backslash N(K)),-\Gamma_0,g(K)-\frac{1}{2})$, which is non-zero. Moreover, since the limit map \ref{eq: direct limit} stabilizes after the first map, the first map $\psi^0_{-,1}$ sends the $Eh(L_K)$ to the LOSS invariant $\mathcal{L}(L_K) \in HFK^-(-Y,K,g(k))$ \cite{etnyre2017sutured}. Thus, we know the Alexander grading of $\mathcal{L}(L_K)$ equals to $g(K)$. \cite[Corollary 1.7]{OzsvathStipsiczContactsurgeryandtransverseinvariant} implies the self linking number of the transverse push-off of $L_K$ is $2g-1$, which realizes the Bennequin bound in $(Y,\xi')$. Hence, Theorem 1.1 in \cite{EtnyreVanHMFiberedBennequin} implies $\xi'=\xi_K$ up to adding Giroux torsion. 
Moreover, by the Theorem 1.5 in \cite{EtnyreVanHMFiberedBennequin}, the Giroux torsion is adding in the complement of $L_K$, so we still have a $\tb=0$ representative in $\xi_K$, which finishes the proof. \epf

To prove Theorem \ref{thm: tb=0 for 2g-1 in other 3-manifold}, which is the general version of Theorem \ref{thm: tb=0 for knot realizeing T-B bound}, we need a general version of Lemma \ref{lem: contact element on top grading}. Note that even though we call it is the general version, the two proofs are identically the same.

\blem\label{lem: general contact element on top grading}
Suppose a null-homologous knot $K$ has a Legendrian representative $L$ realizing the three-dimensional Thurston-Bennequin bound in $(Y,\xi)$ and the contact invariant $c(\xi)\neq 0$. Assume without loss of generality that $L$ has $tb = -n$ for a large positive integer $n$. Let $N(L)$ be a standard contact neighborhood of $L$ and $\xi_L$ be the restriction of $\xi$ on $Y\backslash N(L)$. Then
\[
c(\xi_L) \in SFH(-Y\backslash N(K),-\Gamma_n,g(K)+\frac{n-1}{2})
\]
Furthermore, we have
\[
\pi^n(c(\xi_L)) = c(\xi) \in SFH (-Y(1)).
\]
\elem

\begin{proof} [Proof of Theorem \ref{thm: tb=0 for 2g-1 in other 3-manifold}]

The proof is the same as the proof of Theorem \ref{thm: restated 1.1}. The only thing we need to note is that in $(S^3,\xi_{std})$, it is enough to construct $\xi_0$ such $\pi^0(\xi_0)=\pi^n(c(\xi_L)) = c(\xi_{std})\neq 0$, because there is a unique tight contact structure in $S^3$ with non-vanishing invariant. There is no ambiguity of the contact structure we obtained from $\xi_0$ after we gluing back the contact handles. This argument stays the same when we assume $(Y,\xi)$ is uniquely represented.

If we again have a contact structure $\xi_0$ on $(Y\backslash N(K),\Gamma_0)$ satisfying $\pi^0(\xi_0)=\pi^n(c(\xi_L)) = c(\xi)$, then since $(Y,\xi)$ is unique represented, we know the contact structure obtained from $\xi_0$ is tight and equal to $\xi$. The rest proof is identical.
\end{proof}




\bibliographystyle{alpha}
\bibliography{ref.bib}

\end{document}